\documentclass[11pt]{amsart}
\usepackage[margin=3cm]{geometry}
\usepackage{xcolor}
\newtheorem{Thm}{Theorem}
\newtheorem*{Thm*}{Theorem}
\newtheorem{Remark}{Remark}
\newtheorem*{Def}{Definition}
\newtheorem{conj}{Conjecture}
\newtheorem{lem}{Lemma}
\newtheorem{cor}{Corollary}
\newtheorem{pro}{Proposition}
\newtheorem{Problem}{Problem}
\newcommand\pieter[1]{{\textcolor{blue}{#1}}}

\newcommand{\Ocal}{O}
\renewcommand{\phi}{\varphi}

\def\cT{{\mathcal T}}
\DeclareMathOperator{\mmod}{mod}
\DeclareMathOperator{\li}{li}
\global\long\def\epsilon{\varepsilon}
\newcommand*{\bfrac}[2]{\genfrac{}{}{0pt}{}{#1}{#2}}

\begin{document}
\date{\today}

\title{Constrained ternary integers}
\author{Florian Luca, Pieter Moree, Robert Osburn, Sumaia Saad Eddin and Alisa Sedunova}

\maketitle
{\def\thefootnote{}
}

\begin{abstract}
\noindent 
An integer $n$ is said to be ternary if it is composed of
three distinct odd primes. In this paper, we asymptotically count
the number of ternary integers $n\leq x$ with 
the constituent primes satisfying
various constraints. 
We apply our results to the study of 
the simplest class
of (inverse) cyclotomic polynomials that can have coefficients 
that are greater
than 1 in absolute value, namely 
to the $n^{\text{th}}$ (inverse) cyclotomic
polynomials with ternary $n$. We show, for example, that
the corrected Sister Beiter conjecture is true for
a fraction $\ge 0.925$ of ternary integers.
\end{abstract}

\section{Introduction}
\noindent Let $\omega(n)$ denote the number of distinct prime factors in the prime factorisation
of $n$ and let $\Omega(n)$ be the total number of prime factors. Put 

$$\pi(x,k)=\sum_{n\le x,\; \omega(n)=k}1{\rm \quad and \quad }N(x,k)=\sum_{n\le x,\; \Omega(n)=k}1.$$
Note that $\pi(x,1)$ counts the number of primes $p\le x$. As is usual, 
we will write $\pi(x)$ instead of $\pi(x,1)$.
\par In \cite{L} Landau, confirming a conjecture of Gauss, showed that as $x \to \infty$

\begin{equation} 
\label{landau}
\pi(x,k)\sim N(x,k)\sim {x\over \log x}{(\log \log x)^{k-1}\over (k-1)!} .
\end{equation}
This result is a generalization of the Prime Number Theorem, which is the case $k=1$.
\noindent Nowadays, using the Selberg-Delange method, much more precise estimates can be given
(see e.g. Tenenbaum \cite[pp. 200--206]{Tenen}). In particular, 
we have 
\begin{equation} 
\label{landi2}
\pi(x,k)={x\over \log x}{(\log \log x)^{k-1}\over (k-1)!}\left(1+o_k\left({1\over \log \log x}\right)\right),
\end{equation}
and a similar estimate holds for $N(x,k)$. 
Various authors considered the related problem where  $k$ is allowed to vary to 
some extent with $x$. For a nice survey, see Hildebrand \cite{H}.
\par In this paper, we establish some variations 
of the result of Landau in case $k=3$ (see  Section \ref{resi}), 
which might
be of some interest for cryptography, but certainly have some applications in
the theory of coefficients of cyclotomic polynomials (see Section \ref{apps}). 
Here, in particular, \emph{ternary integers} are of importance.
\begin{Def}
An integer $n$ is said to be ternary if it is of the form
$n=pqr$ with $3\le p<q<r$ primes. It is constrained if on
at least one of $p,q$ and $r$ a constraint is imposed.
\end{Def}
Let $N_T(x)$ denote the number of ternary $n\le x$, that is the number of integers up to $x$ consisting of exactly $3$ different odd prime factors. It is an easy 
consequence (see Corollary \ref{cor1}) of
the validity of the estimate in \eqref{landi2} for $N(x,k)$
that asymptotically
\begin{equation}
\label{moreprecise}
N_T(x)=\frac{x(\log \log x)^2}{2\log x}\Big(1-\frac{(1+o(1))}{\log \log x}\Big).
\end{equation}
\section{Results on constrained ternary integers}
\label{resi}
The theory of ternary (inverse) cyclotomic coefficients
naturally leads to some questions in analytic
number theory. For the sake of brevity we consider only
a few of those. Their applications are discussed in
Section \ref{sec:analytic}.
\begin{Thm}
\label{thm:assT}
Let $p,q,r$ be primes. Put
$$
{\mathcal T(x)}=\left\{pqr \leq x: 3\le p<q<r<\left(\frac{p-1}{p-2}\right)(q-1),~~r\equiv q\equiv \pm 1(\mmod p)\right\}.
$$
We have
$$
|\mathcal{T}(x)|=C_1\frac{x}{(\log x)^2} +O\left(\frac{x\log\log x}{(\log x)^3}\right),
$$
where
\begin{equation}
\label{logsum}
C_1=4\sum_{p\ge 3} \frac{1}{p(p-1)^2} \log\left(\frac{p-1}{p-2}\right)=0.249029016616718\ldots
\end{equation}
\end{Thm}
The terms of the sum $C_1$ are $O(p^{-4})$ and this allows one to obtain $C_1$ with
the indicated precision by truncation at a sufficient large $p$.
\par Theorem \ref{thm:assT} can be applied to obtain analytic results on ternary inverse cyclotomic coefficients, see Theorem \ref{co} in Section~\ref{subsection7.4.1}. 
Note that 
for $x\ge 561$ the smallest integer in ${\mathcal T}(x)$ is 561, which is also the smallest Carmichael number.
\begin{Thm} \label{thm::mod_pq}
Let $a$ be a non-zero integer and $p,q,r$ be distinct odd primes.
Define
$$\cT_a(x) =\{pqr \leq x: \; 3 \leq p<q<r,~r \equiv a (\mmod pq)\}.$$
Then
$$\left| \cT_a(x) \right| = C_2 \frac{x}{\log x} + O\left(\frac{x\log \log x}{(\log x)^2}\right),$$
where 
\begin{equation}
\label{ceetwee}
C_2 = \left( \sum_p \frac{1}{p(p-1)}\right)^2=0.597771234896174\ldots
\end{equation}
\end{Thm}
Here the convergence of the prime sum is much poorer. However, it is 
easily related to zeta values at integer arguments, see \cite[p. 230]{gaps}, 
and in this
way one obtains
$$\sum_p \frac{1}{p(p-1)} = \sum_{k=1}^{\infty}\frac{(\varphi(k)-\mu(k))}{k}\log \zeta(k)
=0.77315666904975\ldots.$$
\par Theorem~\ref{thm::mod_pq} allows one to deduce asymptotic results on the flatness
of ternary cyclotomic polynomials, see 
Theorem~\ref{thm:flatness} in Section~\ref{subsection7.4.2}.
\begin{Thm}
\label{main}
For every odd prime $p\ge 3$ let
$$
M(p)=\{(a_i(p),b_i(p))~:~1\le a_i(p),b_i(p)\le p-1\} 
$$
be a set of mutually distinct pairs $(a_i(p),b_i(p))$
of cardinality
$$|M(p)|=\alpha p^2+O(p),\qquad \text{ as } \;\; p\rightarrow \infty,$$ with 
$0<\alpha<1$.
Put 
$$
{\mathcal T}_M=\{pqr~:~ 3 \leq p<q<r,~(q,r)\equiv (a_i(p),b_i(p))~(\mmod p),~1\le i\le |M(p)|\},
$$
where $p,q$ and $r$ are distinct odd primes.
Then 
\begin{equation*}
{\mathcal T}_M(x)=\frac{\alpha 
x(\log \log x)^2}{2\log x} \left(1+O\left(\frac{1}{\log \log \log x}\right)\right).
\end{equation*}
\end{Thm}
Finally, Theorem~\ref{main} can 
be used to provide further evidence 
of the truth of the corrected Sister Beiter conjecture, see Theorem~\ref{cycchal} in Section~\ref{subsection7.4.3}.
\section{Auxiliary results}
For a positive integer $k$ and a positive real number $x$ we write $\log_k x$ for 
the iteratively defined function given by $\log_1 x=\max\{1,\log x\}$, where $\log x$ is a natural logarithm 
of $x$,  and for $k\ge 2$, $\log_k x=\max\{1,\log_{k-1} x\}$.\\
\indent We first briefly recall some 
standard tools.
\par Chebychev showed that
\begin{equation}
\label{Cestimates}
\pi(x)\asymp \frac{x}{\log x}.
\end{equation}
Since the times of Chebychev our understanding of $\pi(x)$ has
much improved:
\begin{Thm}[Prime Number Theorem in strongest 
form]
\label{PNT}
There exists $c>0$ such that
\[
\pi(x)=\li(x)+O\left(x e^{-c
(\log x)^{3/5}(\log \log x)^{-1/5}}\right),
\]
where $\li(x)$ denotes the logarithmic integral
\[
\li(x)=\int_2^x \frac{dt}{\log t}.
\]
\end{Thm}
The error term above was established in \cite{ford} using the strongest available version of the zero-free region for $\zeta$-function due to Vinogradov and Korobov. It was shown by Trudgian \cite{trudgian} that one can take $c=0.2098$. 

\begin{Thm}[Mertens]
\label{mertens}
We have
$$
\sum_{p\le x} \frac{1}{p}=\log \log x+c_0+O\left(\frac{1}{\log x}\right),
$$
valid for all $x\ge 3$ with some constant $c_0$.
\end{Thm}
\begin{Thm}[Siegel-Walfisz] 
\label{SW}
Given any $A>0$, there exists a constant $c_1(A)$ such that if $d\le \log^A x$, then 
$$ \pi(x;d,a)=\frac{{\rm Li}(x)}{\varphi(d)}+O(xe^{-c_1(A)\sqrt{\log x}}),$$ where
$\pi(x;d,a)=|\{p\le x:p\equiv a(\mmod{d})\}|$.
\end{Thm}

\begin{lem}
\label{lem:prime}
Put $y:=\exp(\log x/\log_2 x)$ and $z_1:=\exp((\log x)^{1/\log_3 x})$. Then there 
exists a positive constant $A$ such that if $z_1<p$ and $p^{\log_2 x}<t\le y$, then
\begin{equation}
\label{eq:prime}
\pi(t;p,a)=\frac{\pi(t)}{p-1} \left(1+O\left(\frac{1}{(\log t)^A}\right)\right)
\end{equation}
holds for all residue classes $a\in \{1,\ldots,p-1\}$  and all $t$ except for at most $2\log_2 x$  exceptional primes $p$ each of which exceeds $\log_2 x$. 
\end{lem}
\noindent {\bf Remark.} Observe that since $t>z_1$, it follows that $(\log t)^A>\log_2 x$ holds for all $x$ sufficiently large. Thus, we may assume that also the error in the estimate of the above lemma (uniformly in our range for $t$), is larger than $\log_2 x$.

\begin{proof}
We follow the proof of Linnik's theorem from  page 54 in \cite{Bom}. Let $p\in (z_1,y^{1/\log_2 x})$ be fixed 
and let $t>p^{\log_2 x}$. There it is shown that if $p\le T$ is any modulus then
$$
\sum_{\substack{q\le t\\ q\equiv a(\mmod p)}} \log q=\frac{t}{\phi(p)}+E+O\left(t^{1/2}+\frac{t\log t}{T}\right),
$$
and
$$
E=-\chi_1(a)\frac{t^{\beta_1}}{\beta_1}+O\left(\frac{F}{\phi(p)} \right),
$$
with
both $E$ and $F$ 
certain sums over zeros of 
$L$ functions $L(s,\chi)$, 
where $\chi$ 
runs  over the characters modulo $p$.
The term $-\chi_1(a) t^{\beta_1}/\beta_1$ appears only if there exists an exceptional zero relative to the pair $(T, c_1)$. For us, we put $T:=t^{2/\log_2 x}$ and take any $c_1$. Then $p\le T^{1/2}$.
If there is an exceptional zero with respect to the pair  $(T,c_1)$, then it is unique. Further, it is also exceptional for the pair  $(T',c_1/2)$ for any $T' \in [T,T^2]$, and 
the prime $p$ satisfies
$$
p>(\log(T^{1/2}))^{c_2}=(\log T)^{c_2/2}.
$$ 
Since $p>z_1$, we have that $t>z_1^{\log_2 x}$, so 
$$
\log t>(\log_2 x)z_1=(\log _2 x) (\log x)^{1/\log_3 x}>(\log_2 x)^2\qquad {\text{\rm for}}\qquad x>x_0.
$$
Hence,  
$$
\log T=\frac{2\log t}{\log_2 x}>(\log t)^{1/2}
$$
uniformly for all our $t$ when $x>x_0$, so $p>(\log T)^{c_2/2}>(\log t)^{c_2/4}$. Note that since $t>p^{\log_2 x}>z_1^{\log_2 x}$, it follows easily that 
$$
(\log t)^{c_2/2}>((\log_2 x) (\log x)^{1/\log_3 x})^{c_2/2}>\log_2 x
$$
for all $x>x(c_1)$. Let us give an upper bound on the number $k$ of exceptional primes of this type. Since we just said that if there is some exceptional prime for $T$, then it is also the exceptional prime for all $T'\in [T,T^2]$, it follows that if we take 
$t_1:=z_1^{\log_2 x}$, $t_2:=t_1^2,~t_3:=t_2^2,~\ldots,~t_k:=t_{k-1}^2$, where $k$ is the smallest positive integer such that $t_k\ge y $, then there can be at most $k$ exceptional primes altogether. Clearly, from the above recurrence we have $t_j=t_1^{2^j}$. Hence,
$$
y\le t_1^{2^k}=(z_1^{2\log_2 x})^{2^k},
$$
and upon taking logarithms we get
$$
\frac{\log x}{\log_2 x}\le 2^k (\log_2 x) (\log x)^{1/\log_3 x},
$$
and taking logarithms once again we get
$$
k\log 2-\frac{\log_2 x}{\log_3 x}\ge \log_2 x-2\log_3 x.
$$
Hence, 
$$
k=\left(\frac{1}{\log 2}+O\left(\frac{1}{\log_3 x}\right)\right)\log_2 x,
$$
so clearly, $k<2\log_2 x$ for all $x$ large enough. From now on, we discard the exceptional primes and work with the remaining ones. For them, 
$$
E=O\left(\frac{F}{\phi(p)}\right),
$$
where by arguments from the middle of page 55 in \cite{Bom} together with the fact that we are under the assumption that there is no exceptional zero,  $F$ is bounded as
$$
F\ll t^{1/2} T^5+\frac{(\log t)t^{1-c_1/\log T}}{\log(t/T^{c_3})}\quad {\text{\rm if}}\quad  t>T^{c_3}.
$$
For us, the inequality $t>T^{2c_3}$ holds for all $x>x_0$, so $\log(t/T^{c_3})\gg \log t$. Further, since in fact $\log T\le 2\log t/\log\log x\le 2\log t/\log\log t$, it follows that $1-c_1/\log T\ge 1-2c_1(\log\log t)/\log t$, therefore 
the second term on the right above is  
$$
\ll \frac{t}{(\log t)^{2c_1+1}}.
$$
Putting everything together, we get that  
\begin{equation}
\label{eq:1}
\sum_{\substack{q\le t\\ q\equiv a(\mmod{p})}} \log q=\frac{t}{\log q}+O\left(t^{1/2}+\frac{t\log t}{T}+\frac{t^{1/2} T^5}{\phi(p)}+\frac{t}{\phi(p) (\log t)^{2c_1+1}}\right).
\end{equation}
Since $\phi(p)<p\le T^{1/2}=t^{o(1)}$, the first and third terms above are all dominated by the fourth term, while the second one is
$$
\frac{t\log t}{T}.
$$
It remains to show that this is also dominated by the fourth one. Since $T^{1/2}\ge p>\phi(p)$, it suffices to show that
$$
T^{1/2}>(\log t)^{2c_1+2}.
$$
This is equivalent to 
$$
\frac{\log t}{\log_2 x}>(2c_1+1)\log_2 t, \qquad {\text{\rm or}}\qquad \frac{\log t}{\log_2 t}>(2c_1+1)\log_2 x.
$$
The function $t\mapsto \log t/\log_2 t$ is increasing for $t>e^e$, and since for us $t>z_1^{\log_2 x}>z_1$, we have 
$$
\frac{\log t}{\log_2 t}>\frac{(\log x)^{1/\log_3 x}}{((\log_2 x)/\log_3 x)}
$$
and the function on the right hand side exceeds any multiple of $\log_2 x$ for $x$ sufficiently large. 
Hence, the four contributions to the
error term in \eqref{eq:1} are dominated by the fourth one, and so
$$
\sum_{\substack{q\le t\\ q\equiv a(\mmod{p})}} \log q=\frac{t}{\phi(p)}\left(1+O\left(\frac{1}{(\log t)^A}\right)\right),
$$
where we can take $A=2c_1$. This is uniform for all $t$ in our range, and now the  desired conclusion 
follows by Abel summation. 
 \end{proof}

Let $k\ge 1$.
Put $$M(x,k)=\sum_{n\le x,~\Omega(n)=k}\mu(n)^2.$$ 
Using elementary analytic
number theory an asymptotic for
$M(x,k)$ can be
derived (cf. 
Hardy and Wright \cite[Theorem 437]{HW}). We will need a slightly
stronger result.
\begin{lem}
\label{mxk}
We have
\begin{equation*}
M(x,k)={x\over \log x}{(\log \log x)^{k-1}\over (k-1)!}
\left(1+o_k\left({1\over \log \log x}\right)\right).
\end{equation*}
\end{lem}
\begin{proof}
As remarked in the introduction one has the estimate
\begin{equation} 
\label{nxk}
N(x,k)={x\over \log x}{(\log \log x)^{k-1}\over (k-1)!}\left(1+o_k\left({1\over \log \log x}\right)\right).
\end{equation}
For $k=1$ the result is merely a weaker variant
of Theorem \ref{PNT}, the Prime Number Theorem.
For $k\ge 2$ the idea of the proof is to relate $M(x,k)$ to $N(x,k)$ and
use the estimate \eqref{nxk}.
Noting that 
$M(x,2)=N(x,2)-\sum_{p\le \sqrt{x}}1$ and using \eqref{nxk} with $k=2$, the claim follows for $k=2$ and so
we may assume that $k\ge 3$.
\par Observe that  if $\Omega(n)=k$, then either $n$ is square-free or $n=p^2 m$ with
$\Omega(m)=k-2$ and $p$ a prime. It follows that 
$$M(x,k)=N(x,k)+O\Big(\sum_{p\le \sqrt{x}}N(\frac{x}{p^2},k-2)\Big).$$
Using the trivial estimate $N(x,k-2)=O(x)$ in the range $x^{1/3}\le p\le \sqrt{x}$ and 
the non-trivial estimate \eqref{nxk} in the range  $p<x^{1/3}$, the proof is easily completed.
\end{proof}
\begin{cor}
\label{cor1}
The counting function $N_T(x)$ satisfies the asymptotic estimate
\eqref{moreprecise}.
\end{cor}
\begin{proof}
Note that $N_T(x)=M(x,3)- M\left(x/2,2\right)+ \pi(x/4)+O(1)$
and use the lemma for $k=3$ and $k=2$.
\end{proof}

\section{The proof of Theorem \ref{thm:assT}}
\begin{proof}[Proof of Theorem \ref{thm:assT}]
We observe that for ternary $n$, 
$$
p^3<n\le x,\quad {\text{\rm therefore}}\quad p<x^{1/3},
$$
and similarly 
$$
pq^2<n\le x,\quad {\text{\rm therefore}}\quad q<\sqrt{x/p}.
$$ 
Thus, 
\begin{equation} \label{firstT}
\begin{split}
|\cT (x)| &= \sum_{3 \leq p < x^{\frac{1}{3}}} \sum_{\bfrac{p < q < \sqrt{x/p}}{q \equiv \pm 1(\mmod p)}} \sum_{\bfrac{q < r \leq \frac{p-1}{p-2}(q-1)}{\bfrac{pqr \leq x}{r \equiv q (\mmod p)}}} 1.
\end{split}
\end{equation}

Denote the inner sum over $r$ by $\sigma_r$. We start with a lower bound on $|\cT (x)|$. Take $p=3$. Then $r \equiv q (\mmod 3)$ and $q<r<2q-2$. Thus, by Theorem~\ref{SW},
$\sigma_r \gg \frac{q}{\log q}$
for $q\ge q_0$. Note also that any such $r$ leads to a legitimate choice for 
$n\in {\mathcal T}(x)$ provided that $3q(2q) \leq x$, that is, whenever $q\le \sqrt{x/6}$. Thus, for $x\ge x_0$
\begin{eqnarray*}
|\cT(x)| \gg \sum_{q_0\le q\le \sqrt{x/6}} \frac{q}{\log q} \gg \int_{q_0}^{\sqrt{x/6}} \frac{t d\pi(t)}{\log t} \gg \frac{t^2}{(\log t)^2}\Big|_{t=2}^{\sqrt{x/6}}\gg \frac{x}{(\log x)^2}.
\end{eqnarray*}
We now determine an asymptotic for ${\mathcal T}(x)$ and show that  $x/(\log x)^2$ is indeed the
correct order of magnitude.

Neglecting the primality condition on $r$ we obtain
\begin{equation}
\label{eq:interval}
\sigma_r \le \pi\left(q-1+\frac{q-1}{p-2}; p;q\right)-\pi(q;p,q)\le \frac{1}{p}\left(\frac{q-1}{p-2}-1\right)+1\ll \frac{q}{p^2}+1.
\end{equation}
We now sum up over all $q$ forgetting the congruence condition on $q$. 
It follows that for a fixed $p$, the number
of constrained ternary integers under scrutiny is of
order at most
\begin{equation}
\label{eq:pisfixed}
\frac{1}{p^2} \left(\sum_{q\le \sqrt{x/p}} q\right)+\pi\left(\sqrt{\frac{x}{p}}\right).
\end{equation}
For us $p<x^{1/3}$, therefore $\log(x/p)\gg \log x$, and
thus the second term in \eqref{eq:pisfixed} is, by the 
Chebychev estimates \eqref{Cestimates}, 
$$
\pi\left(\sqrt{\frac{x}{p}}\right)\ll \frac{\sqrt{x}}{\sqrt{p} \log(x/p)}\ll \frac{\sqrt{x}}{\sqrt{p} \log x}.
$$
For the first term in \eqref{eq:pisfixed} above, we can also use the Chebychev estimates and get that
$$ 
\sum_{q\le \sqrt{x/p}} q\ll \int_{2}^{\sqrt{x/p}} t d\pi(t)\ll \frac{t^2}{\log t}\Big|_{t=2}^{t=\sqrt{x/p}}\ll \frac{x}{p\log(x/p)}\ll \frac{x}{p\log x}.
$$
Thus, for a fixed $p$, the number of choices for $n$ is at most of order 
\begin{equation}
\label{eq:pisfixed2}
\ll \frac{x}{p^3 \log x}+\frac{\sqrt{x}}{\sqrt{p} \log x}.
\end{equation}
We now sum up over $p$. We deal first with the second term in \eqref{eq:pisfixed2}. There, even forgetting that $p$ is prime, we get that this term contributes at most
$$
\frac{\sqrt{x}}{\log x} \sum_{p\le x^{1/3}} \frac{1}{\sqrt{p}}\ll \frac{\sqrt{x}}{\log x} \int_{2}^{x^{1/3}} \frac{dt}{\sqrt{t}}\ll \frac{x^{\frac{1}{2}+\frac{1}{6}}}{\log x}  = \frac{x^{2/3}}{\log x}
$$
to ${\mathcal T}(x).$
Next we deal with the first term in \eqref{eq:pisfixed2}, when summed up over all $p>\log x$. There we get, even forgetting the condition that $p$ is prime, that this term contributes
\begin{equation}
\label{eq:calcu}
\frac{x}{\log x} \sum_{p>\log x} \frac{1}{p^3} \ll \frac{x}{\log x} \int_{\log x}^{\infty} \frac{dt}{t^3} \ll \frac{x}{\log x} \left(-\frac{1}{t^2} \Big|_{t=\log x}^{t=\infty}\right)\ll \frac{x}{(\log x)^3}.
\end{equation}
Thus, (\ref{eq:pisfixed}) is small compared to $|\cT(x)|$ when $p > \log x$. 
We see that the main contribution comes from $p\le \log x$ and from now on, we work under this assumption. 
Let us now go back to \eqref{eq:interval} and assume in addition that $q<\sqrt{x}/ \log x$. Summing up over all primes $q\le \sqrt{x}/ \log x$ of this type, we get instead of \eqref{eq:pisfixed} the number of integers $n\in {\mathcal T}(x)$ of size at most
$$
\frac{1}{p^2} \sum_{q\le \sqrt{x}/ \log x } q+\pi \left(\sqrt{x/p}\right) \ll  \frac{x}{p^2 (\log x)^3},
$$
since $p\le \log x$. Summing up over all $p$, we get a contribution of $O \left(x/(\log x)^3 \right)$ to $|\cT(x)|$, which is small.

So, from now on we work in the range $p\le \log x$ and $\sqrt{x}/\log x <q<\sqrt{x/p}$. One can rewrite (\ref{firstT}) as follows
\begin{equation*}
\begin{split}
|\cT (x)| &= \sum_{3 < p \leq \log x} \sum_{\bfrac{ \frac{\sqrt{x}}{\log x} < q < \sqrt{x/p}}{q \equiv \pm 1(\mmod p)}} \sum_{\bfrac{q < r \leq \min \left(\frac{p-1}{p-2}(q-1), \frac{x}{pq}\right)}{{r \equiv q (\mmod p)}}} 1 + O\left(\frac{x}{(\log x)^3}\right).
\end{split}
\end{equation*}
Write $\sigma_r'$ for the inner sum.
It clearly makes sense for large $x$ and $p\le \log x$ to write $q_p$ for the solution  $q$ of
$$
\frac{x}{pq}=q-1+\frac{q-1}{p-2}=\left(\frac{p-1}{p-2}\right)(q-1).
$$
Hence,
$$q-\frac{1}{2}=\sqrt{\frac{x(p-2)}{p(p-1)}+O(1)}=\sqrt{\frac{x(p-2)}{p(p-1)} \left(1+O\left(\frac{p}{x}\right)\right)}={\sqrt{\frac{x(p-2)}{p(p-1)}}}+O(1),
$$
which gives
\begin{equation}
\label{eq:qp}
q_p=\sqrt{\frac{x(p-2)}{p(p-1)}}+O(1).
\end{equation}
Suppose first that $q\le q_p$. Then, by Theorem~\ref{SW}, we have
\begin{eqnarray*}
\sigma_r' = \frac{\pi(q-1+(q-1)/(p-2))-\pi(q)}{\varphi(p)} + O\left(qe^{-c_0{\sqrt{\log q}}}\right)
\end{eqnarray*}
for some constant $c_0>0$. For us, $\log q=(1/2+o(1))\log x$. Further, by 
Theorem \ref{PNT}  we have that
\begin{eqnarray*}
\pi\left(q-1+\frac{q-1}{p-2}\right)-\pi(q) = \int_q^{q-1+\frac{q-1}{p-2}} \frac{dt}{\log t}+ O\left(q
e^{-c_1(\log q)^{3/5} (\log_2 q)^{-1/5}}\right)
\end{eqnarray*}
for some constant $c_1>0$. Putting everything together, we get that when $p,~q\le q_p$ are fixed
$$
\sigma_r'=\frac{1}{p-1}\int_q^{q-1+\frac{q-1}{p-2}} \frac{dt}{\log t}+O\left(qe^{-c_2{\sqrt{\log q}}}\right)
$$
for some constant $c_2>0$. We split the integral as
$$
\int_q^{q-1+\frac{q-1}{p-2}} \frac{dt}{\log t}=\int_q^{q +\frac{q}{p-2}} \frac{dt}{\log t}+\int_{q+\frac{q}{p-2}}^{q-1+\frac{q-1}{p-2}} \frac{dt}{\log t}.$$
In the second integral, the length of the interval is $O(1)$ and the integral is of size $O(1/\log x)$. Thus, 
\begin{eqnarray}
\label{eq:int1}
\sigma_r' = \frac{1}{p-1}\int_{q}^{q+\frac{q}{p-2}}\frac{dt}{\log t}
+ O\left(qe^{-c_2{\sqrt{\log q}}}\right).\nonumber
\end{eqnarray}
Now we estimate the
latter integral. We make the substitution $t=qu$ for which $dt=qdu$. We get
\begin{equation*}
\begin{split}
\int_q^{q+\frac{q}{p-2}} \frac{dt}{\log t} &= 
\int_1^{1+\frac{1}{p-2}} \frac{qdu}{\log q+\log u}\\
&= \frac{q}{\log q} \int_1^{1+\frac{1}{p-2}} du - \frac{q}{\log q} \int_1^{1+\frac{1}{p-2}} \frac{\log u}{\log q+\log u} du\\
&=\frac{q}{(p-2)\log q} + O \left(\frac{q}{p^2 (\log x)^2}\right).
\end{split}
\end{equation*}
In the last inequality above, we used the fact that 
$$0\le \log u\le \log\left(1+\frac{1}{p-2}\right)\le \frac{1}{p-2}$$ for all $u\in [1,1+1/(p-2)]$.
Further, notice that since $\sqrt{x}/\log x<q<\sqrt{x/p}$, we have that 
$\log q=\frac{1}{2}\log x+O(\log\log x)$ 
and hence,
\begin{eqnarray*}
\frac{1}{\log q}= \frac{2}{\log x} \left(1+O\left(\frac{\log_2 x}{\log x}\right)\right)^{-1} = \frac{2}{\log x}+O\left(\frac{\log_2 x}{(\log x)^2}\right).
\end{eqnarray*}
Thus,
\begin{equation}
\label{eq:11}
\sigma_r'=\frac{2q}{(p-1)(p-2)\log x}+O\left(\frac{q\log_2 x}{p^2(\log x)^2}\right).
\end{equation}
Next consider $q>q_p$. Then certainly $x/(pq)\asymp q$ (in fact, $q_p>\sqrt{x/(4p)}$ for large enough $x$).   So, by the same argument and using Theorems~\ref{PNT} and \ref{SW},
 we have
 \begin{equation}  \label{eq:22}
 \begin{split}
 \sigma_r' &= \frac{\pi\left(x/(pq)\right)-\pi(q)}{\varphi(p)}+O\left(qe^{-c_3{\sqrt{\log x}}}\right)  = 
 \frac{x/(pq)-q}{(p-1)\log q} + O \left(\frac{q}{p(\log x)^2}\right)\\
&= \frac{2}{(p-1)\log x} \left(\frac{x}{pq}-q\right) + O \left(\frac{q \log_2 x}{p (\log x)^2}\right).
 \end{split}
 \end{equation}
Combining \eqref{eq:11} and \eqref{eq:22}, we get  
 $$
 \sigma_r' = \frac{2 a_{p,q}(x)}{(p-1)\log x} + O\left(\frac{q\log_2 x}{p(\log x)^2}\right), \;\;\;\text{ where }\;\;\; a_{p,q}(x)=\left\{\begin{matrix} \frac{q}{p-2} & {\text{\rm if}} & q\le q_p;\\
 \frac{x}{pq}-q & {\text{\rm if}} & q> q_p.
 \end{matrix}\right.
 $$
We sum up over $q$ and first deal with the error term. Since
\begin{equation*}
\begin{split}
\sum_{p\le \log x}\sum_{\substack{q \le \sqrt{x/p}\\  q\equiv \pm 1(\mmod p)}} \frac{q}{p} &\ll \sum_{p\ge 3} \frac{1}{p} \int_3^{\sqrt{x/p}} t \, d\pi(t;p,\pm 1)\cr
&\ll \sum_{p\ge 3} \left(\frac{t^2}{p(p-1)\log t}\Big|_{2}^{\sqrt{x/p}} \right)
\ll \frac{x}{\log x},
\end{split}
\end{equation*}
then the error term coming from $\sigma_r$ is 
$O\left(x(\log x)^{-3} \log_2 x \right)$. 
Thus, we have
\begin{equation} \label{Twithapq}
|\cT(x)| = \sum_{p \leq \log x} \sum_{\bfrac{\frac{\sqrt{x}}{\log x} < q \leq \sqrt{x/p}}{q \equiv \pm 1 (\mmod p)}} \frac{2a_{p,q}(x)}{(p-1)\log x}+O \left(\frac{x\log_2 x}{(\log x)^3} \right).
\end{equation}
It remains to deal with the main term. We let $\varepsilon\in \{\pm 1\}$ and sum over all $q$ in the interval $\sqrt{x}/\log x<q<q_p$ such that $q\equiv \varepsilon (\mmod p)$. By Abel's summation formula, one gets
\begin{eqnarray}
\label{eq:Abel}
\sum_{\substack{\frac{\sqrt{x}}{\log x}<q\le q_p\\ q\equiv \varepsilon (\mmod p)}} q & = & q_p\pi(q_p;p,\varepsilon)-\frac{\sqrt{x}}{\log x}\pi\left(\frac{\sqrt{x}}{\log x}; p,\varepsilon \right) - \int_{\frac{\sqrt{x}}{\log x}}^{q_p} \pi(t;p,\varepsilon) dt.
\end{eqnarray}
By combining Theorem~\ref{PNT} and Theorem~\ref{SW}, we 
obtain that
$$
\pi(t;p,\varepsilon)=\frac{t}{(p-1)\log t}+O\left(\frac{t}{p(\log t)^2}\right)\quad {\text{\rm uniformly~in}}\quad t\in \left[\frac{\sqrt{x}}{\log x},\sqrt{\frac{x}{p}}~\right].
$$
Thus, one can check that 
$$\sum_{\substack{\frac{\sqrt{x}}{\log x}<q\le q_p\\ q\equiv \varepsilon (\mmod p)}} q= \frac{q_p^2}{(p-1)\log x}+O\left(\frac{x\log_2 x}{p(\log x)^2}\right).$$
This was for a fixed $\varepsilon\in \{\pm 1\}$ and for $q\le q_p$. It remains to deal with the contribution of $q$ in the range $q_p< q\le \sqrt{x/p}$. For this, we need to compute
$$
\sum_{\substack{q_p< q\le \sqrt{x/p}\\ q\equiv \varepsilon (\mmod p)}} \left(\frac{x}{pq}-q\right) =
\frac{x}{p}\sum_{\substack{q_p< x\le \sqrt{x/p}\\ q_p\equiv \varepsilon (\mmod p)}} \frac{1}{q}-\sum_{\substack{q_p\le q\le \sqrt{x/p}\\ q\equiv \varepsilon (\mmod p)}} q.
$$
The second sum is, by the above arguments,
\begin{equation*}
\sum_{\substack{q_p< q\le \sqrt{x/p}\\ q\equiv \varepsilon (\mmod p)}} q =
\frac{x}{p(p-1)\log x}-\frac{q_p^2}{(p-1)\log x}+O\left(\frac{x\log_2 x}{p(\log x)^2}\right).
\end{equation*}
 Accounting for the fact that we have two values of $\varepsilon$
 and inserting the above estimates into (\ref{Twithapq}), we get 
 \begin{equation*}
|\cT(x)| = \sum_{p \leq \log x} \left(\frac{2x}{p (p-1) \log x} \right. \sum_{\bfrac{q_p \leq q \leq \sqrt{x/p}}{q \equiv \pm 1 (\mmod p)}} \frac{1}{q} + \left. \frac{4 f(x,p,q_p)}{(p-1)(\log x)^2} \right)+O \left(\frac{x\log_2 x}{(\log x)^3} \right),
\end{equation*}
where
$$f(x,p,q_p) = \frac{q_p^2}{(p-2)(p-1)}-\frac{x}{p(p-1)}+\frac{q_p^2}{p-1}.$$
Using \eqref{eq:qp}, we see that
$$
q_p^2=\frac{x(p-2)}{p(p-1)}+O(q_p)=\frac{x(p-2)}{p(p-1)}+O(\sqrt{x}),
$$
and hence 
 \begin{eqnarray*}
f(x,p,q_p) & = & \frac{x}{p(p-1)^2}-\frac{x}{p(p-1)}+\frac{x(p-2)}{p(p-1)^2} + O\left(\frac{\sqrt{x}}{p}\right)=  O\left(\frac{\sqrt{x}}{p}\right).
 \end{eqnarray*}
Thus, the contribution coming from the sum over $p$ to $\cT(x)$ of 
the term that contains $f(x,p,q_p)$, is 
$$O\left(\frac{\sqrt{x}}{(\log x)^2} \right. \sum_{p \leq \log x} \left. \frac{1}{p(p-1)}\right),$$ 
which is small. We 
conclude that
$$|\cT(x)| =  \frac{2x}{\log x}\sum_{p \leq \log x} \frac{1}{p (p-1)} \sum_{\bfrac{q_p \leq q \leq \sqrt{x/p}}{q \equiv \pm 1 (\mmod p)}} \frac{1}{q} + O \left(\frac{x\log_2 x}{(\log x)^3} \right).$$
Using again the Abel summation formula we get (after a short computation) that
for a fixed $\epsilon \in \{\pm 1\}$,
\begin{equation*}
\sum_{\substack{q_p\le q\le \sqrt{x/p}\\ q\equiv \varepsilon (\mmod  p)}} \frac{1}{q}
=
\frac{1}{p-1}\frac{\log (\frac{p-1}{p-2})}{\log x}+O\left(\frac{\log_2 x}{p(\log x)^2}\right).
\end{equation*}
Since there are two values for $\varepsilon$, the contribution of a fixed $p$ to the number of elements of ${\mathcal T}(x)$ is
$$
\frac{4}{p(p-1)^2 (\log x)^2} \log \left(\frac{p-1}{p-2}\right)+O\left(\frac{x\log_2 x}{p^3(\log x)^2}\right).
$$
We now sum over $3\le p\le \log x$, getting
$$
\frac{4}{(\log x)^2} \left(\sum_{p\le \log x} \frac{1}{p(p-1)^2} \log \left(\frac{p-1}{p-2}\right)\right)+O\left(\frac{x \log_2 x}{(\log x)^3} \sum_{p\ge 3} \frac{1}{p^3}\right).
$$
The error term is $O(x\log_2 x/(\log x)^3)$. As for the main term, we can take the sum of the series to infinity introducing a tail of size
$$
\sum_{p>\log x} \frac{1}{p(p-1)^2}\log \left(\frac{p-1}{p-2}\right)\ll \sum_{m>\log x} \frac{1}{m^4}\ll \frac{1}{(\log x)^3}.
$$
The result is therefore proved.
\end{proof}

\section{Proof of Theorem \ref{thm::mod_pq}}
\begin{proof}[Proof of Theorem \ref{thm::mod_pq}]
We proceed as in the proof of Theorem~\ref{thm:assT}.
Since $p^3 < pqr \leq x$, 
it follows that $p < x^{\frac{1}{3}}$ and similarly $p q^2 < pqr \leq x$ implies $q < \sqrt{x/p}$.
Thus, we want to count
\begin{equation}
\label{eq01}
\left| \cT_a(x) \right|= \sum_{p \leq x^{\frac{1}{3}}} \sum_{p < q < \sqrt{x/p}} \sum_{\bfrac{q<r \leq x/(pq)}{r \equiv a (\mmod pq)}} 1.
\end{equation}
Let $p=3$ and $q=5$. Then $r$ runs over 
some arithmetic progression modulo $15$ in the range $5 < r \leq x/15$. By Theorem~\ref{SW}, it follows that $|\cT_a(x)| \gg x/\log x$.

We denote the inner sum over $r$ in~\eqref{eq01} by $\sigma_r$. By neglecting the condition of $r$ being prime we obtain
$$\sigma_r = \sum_{\bfrac{q<r \leq x/(pq)}{r \equiv a (\mmod pq)}} 1 \leq \frac{1}{pq} \left(\frac{x}{pq}-q\right) = \frac{x}{(pq)^2} -\frac{1}{p}.$$
Thus,
\begin{equation*}
\begin{split}
\left| \cT_a(x) \right| &= x \sum_{p \leq x^{\frac{1}{3}}} \frac{1}{p^2}\sum_{p < q < \sqrt{x/p}} \frac{1}{q^2} - 
\sum_{p \leq x^{\frac{1}{3}}} \frac{1}{p} \sum_{p < q < \sqrt{x/p}} 1.\\
\end{split}
\end{equation*}
Define
$$\cT'_a(x) =\{pqr \leq x: \; 3 \leq p<q<r,~r \equiv a (\mmod pq),~g\geq (\log x)^2\}.$$
 Let $\cT'_a(x)$ count the integers counted by
$\cT_a(x)$ with the additional requirement that
$q \geq (\log x)^2$.
We then have 
\begin{equation*}
\begin{split}
\left| \cT'_a(x) \right| &< \frac{x}{(\log x)^2} \sum_{p \leq x^{\frac{1}{3}}} \frac{1}{p^2}\sum_{p < q < \sqrt{x/p}} \frac{1}{q}
<   \frac{x\log_2 x }{(\log x)^2} \sum_{p \leq x^{\frac{1}{3}}} \frac{1}{p^2} 
\ll  \frac{x\log_2 x}{(\log x)^2},
\end{split}
\end{equation*}
where we used Theorem \ref{mertens}.
Similarly if $p \geq (\log x)^2$, then we can improve the bound to
\begin{equation*}
\begin{split}
\left| \cT'_a(x) \right| \ll  \frac{x\log_2 x}{(\log x)^4}.
\end{split}
\end{equation*}
By the above we get
$$\left|\cT_a(x)\right| =  \sum_{p < (\log x)^2} \sum_{\bfrac{p < q < \sqrt{x/p}}{q < (\log x)^2}} \sigma_r + O \left(\frac{x\log_2 x}{(\log x)^2}\right).$$
On noticing that $\pi(q;a,pq)=\pi(q)$, we obtain
$$\sum_{p < (\log x)^2}\sum_{p < q < (\log x)^2} \pi \left(q\right) \ll \sum_{p < (\log x)^2} \int_{p}^{(\log x)^2} t d \pi(t) \ll \frac{(\log x)^6}{(\log_2 x)^2}.$$
We then write
\begin{equation*}
\sigma_r = \pi \left(\frac{x}{pq};a,pq\right)-\pi(q;a,pq),
\end{equation*}
and get
\begin{equation*}
\begin{split}
\left|\cT_a(x)\right| &=  \sum_{p < (\log x)^2}\sum_{p < q < (\log x)^2} \pi \left(\frac{x}{pq};a,pq\right) + O \left(\frac{x\log_2 x}{(\log x)^2}\right).\\
\end{split}
\end{equation*}
Since $\log (x/pq) = \log x + O(\log_2 x)$, the main term above equals
\begin{equation*}
\begin{split}
x &\sum_{p < (\log x)^2} \frac{1}{p(p-1)}\sum_{p < q < (\log x)^2} \frac{1}{q(q-1)} \frac{1}{\log \left(\frac{x}{pq}\right)} \\
= \frac{x}{\log x} &\sum_{p < (\log x)^2} \frac{1}{p(p-1)}\sum_{p < q < (\log x)^2} \frac{1}{q(q-1)} + O \left(\frac{x\log_2 x}{(\log x)^2}\right).
\end{split}
\end{equation*}
We complete the sums above to infinity with an error of a suitable size and get 
$$|\cT_a(x)| = C_2 \frac{x}{\log x} + O \left( \frac{x\log_2 x}{(\log x)^2}\right),$$
thus concluding the proof.


\end{proof}

\section{The proof of Theorem \ref{main}}
\noindent Note that there are $(p-1)^2$ possible pairs of residue classes $(a,b)$ modulo $p$ with $1\le a,b\le p-1$. Recall that
\begin{equation}
\label{pi3x}
N_T(x)=|\{n=pqr\le x: 3\le p<q<r\} |\sim \frac{x(\log_2 x)^2}{2\log x}.
\end{equation}
Hence, by restricting for each $p$ the number of possibilities of the pair  $(q,r)$ modulo $p$ to 
a fraction $\alpha$ of the total number of possibilities, we end up with a set of positive integers
the cardinality of which, if we count them up to $x$, is asymptotic to $\alpha$ times the total number of positive integers $n\le x$ with exactly three prime factors $p<q<r$. 
Notice that a comparison of Theorem~\ref{main} with \eqref{pi3x} shows
that this simple heuristic idea is 
actually true.
\par For ease of exposition in the proof of Theorem~\ref{main}, we now let
$$y:=\exp\Big(\frac{\log x}{\log_2 x}\Big),~z_1:=\exp\Big(\exp\Big(\frac{\log _2 x}{\log_3 x}\Big)\Big),~y_1:=\exp\Big(\frac{\log x}{\exp((\log_3 x)^2)}\Big).$$
\begin{proof}[ The proof of Theorem 
\ref{main}] Let $n=pqr\le x$ with $p<q<r$. Then
$$
p^3<x\quad {\text{\rm and}}\quad  pq^2<x,
$$
and so
$$
p<x^{1/3}\quad {\text{\rm and}}\quad q<{\sqrt{{x}/{p}}}.
$$
We may also assume that $n>x/\log x$, since otherwise there are at most $O(x/\log x)$ integers $n\le x$, regardless of the number of their prime factors. Thus, 
$$
\frac{x}{pq\log x}<r\leq \frac{x}{pq}.
$$
Furthermore, $r^3>n>x/\log x$, so $r>(x/\log x)^{1/3}$. Fix $p$ and $q$. Since $r \leq x/(pq)$, the number of possibilities for $r$ (disregarding the congruence conditions on $(q,r)$ modulo $p$) is less or equal than
\begin{equation}
\label{eq:111}
\pi\left(\frac{x}{pq}\right)\ll \frac{x}{pq\log(x/pq)}\ll \frac{x}{pq\log x},
\end{equation}
where for the last inequality we used the fact that 
$$\frac{x}{pq} \geq r>\left(\frac{x}{\log x}\right)^{1/3}\gg x^{1/4},\quad {\text{\rm so}}\quad  \log(x/pq)\gg \log x.
$$
Assume $q\in [y,x]$. Then for a fixed $p$, the number of $n\le x$ with such $q$  is by Theorem~\ref{mertens} of order at most
\begin{equation}
\label{eq:2}
\frac{x}{p\log x}\sum_{y<q<x}\frac{1}{q}\ll \frac{x}{p\log x} (\log_2 x-\log_2 y+o(1))\ll \frac{x\log_3 x}{p\log x}.
\end{equation}
Summing up \eqref{eq:2} over all $p\le x^{1/3}$, we get an upper bound of
$$
\frac{x\log_3 x}{\log x}\sum_{p\le x^{1/3}} \frac{1}{p}\ll \frac{x\log_2 x\log_3 x}{\log x}=O\left(N_T(x) \frac{\log_3 x}{\log_2 x}\right)
$$
on the set of such $n\le x$. So, from now on we may assume that $q\le y$. Assume that $p\le z_1$. Then summing up \eqref{eq:111} over all $p\le z_1$ but $q$ fixed, we get a number of $n\le x$ of order
$$
\frac{x}{q\log x} \sum_{p\le z_1} \frac{1}{p}\ll \frac{x}{q\log x} (\log_2 z_1+O(1))\ll 
\frac{x}{q\log x}\frac{\log_2 x}{\log_3 x}.
$$
Summing up the above inequality over all $q\le \sqrt{x}$, we get an upper bound of order
$$
\frac{x}{\log x}\frac{\log_2 x}{\log_3 x}\sum_{q\le \sqrt{x} }\frac{1}{q}\ll \frac{x}{\log x}\frac{(\log_2 x)^2}
{\log_3 x}=O\left(\frac{N_T(x)}{\log_3 x}\right)
$$
on the set of such $n\le x$, so we can ignore such $n$. 
So, from now on $z_1<p<q<y$. Assume next that $q<p^{\log_2 x}$. Then $p<q<p^{\log_2 x}$. Keeping $p$ fixed and summing up inequality \eqref{eq:111} over all such $q$ we get that the number of integers $n\le x$ is of order at most
$$
\frac{x}{p\log x} \sum_{p<q<p^{\log_2 x}}\frac{1}{q}\ll \frac{x}{p\log x}(\log_2 (p^{\log_2 x})-\log_2 p+O(1))\ll \frac{x\log_3 x}{p\log x}.
$$
Summing up over all $p\le x^{1/3}$, we get that the total number of $n\le x$ is of order at most
$$
\frac{x\log_3 x}{\log x} \sum_{p\le x^{1/3}} \frac{1}{p}\ll \frac{x\log_2 x\log_3 x}{\log x}=O\left(N_T(x) \frac{\log_3 x}{\log_2 x}\right),
$$
and this is negligible for us. So, we can ignore such integers $n$ from our argument. 
So, from now on, we may assume that $p^{\log_2 x}<q$. Since also $q<y$, it follows that $p<y^{1/\log_2 x}=\exp(\log x/(\log_2 x)^2)$. In fact, we will do better. We assume that $n$ is such that $y_1\le p<x^{1/3}$. Then keeping $q$ fixed and summing over such $p$, we get a number of such $n$ of order at most
$$
\frac{x}{q\log x}\sum_{y_1\le p\le x^{1/3}} \frac{1}{p}\ll \frac{x}{q\log x} (\log_2 x^{1/3}-\log_2 y_1)\ll \frac{x(\log_3 x)^2}{q\log x}.
$$
Summing up the above bound over all $q\le y$, we get a bound of 
$$
\frac{x(\log_3 x)^2}{\log x} \sum_{q\le y} \frac{1}{q}\ll \frac{x(\log_2 x)(\log_3 x)^2}{\log x}=O\left(N_T(x) \frac{(\log_3 x)^2}{\log_2 x}\right)
$$
for the number of such $n\le x$, and this is negligible for us. So, we may assume that $p\in [z_1,y_1]$.

We plan to apply Lemma \ref{lem:prime}. We deal first with the exceptional primes. Let $P_E$ be the set of such primes. Recall that
$p>\log_2 x$ and $\#P_E\le 2\log_2 x$ 
by Lemma \ref{lem:prime}. Fixing $p\in P_E$, the remaining
$qr\le x/p$ can be chosen in at most
$$
\pi_2\left(\frac{x}{p}\right)\ll \frac{x}{p}\frac{\log_2(x/p)}{\log(x/p)}
\ll \frac{x\log _2 x}{p \log x}
$$
ways. Here we used the fact that 
$p^{\log_2 x}<y<x$ and so
$p<x^{1/\log_2 x}$, which implies that 
$\log(x/p)\gg \log x$. Now $p$ is in a set of at most $2\log_2 x$ elements each larger than $\log_2 x$.
We now sum up over $p\in P_E$. Discarding the information that they are primes and keeping only the information about their sizes and the number of them, we get a contribution of at most
$$
\frac{x\log_2 x}{\log x}\sum_{\substack{p\in P_E\\ \log_2 x <p\\ \# P_E\le 2\log_2 x}} \frac{1}{p}
\ll \frac{x\log_2 x}{\log x} \left(\frac{\#P_E}{\log_2 x}\right)\ll \frac{x\log_2 x}{\log x}=O\Big(\frac{N_T(x)}{\log_2 x}\Big)
$$
ternary integers, and we are done.

Now we are in a situation were
we can apply Lemma \ref{lem:prime}. We may assume that the estimate \eqref{eq:prime} holds for all $p\in [z_1,y_1]$ and all $t$ such that $p^{\log_2 x}<t\le y$. 
So, we fix $p$ in our range. 
We fix pair of residue classes $(a,b)\in \{1,\ldots,p-1\}$ such that $(a,b)\in M(p)$.  We also fix $q$ in the interval
$(p^{\log_2 x}, y]$ such that 
$q\equiv a(\mmod{p})$. So, we need to count the number of primes $$
r\in \left[\frac{x}{pq(\log x)},\frac{x}{pq}\right]
$$
which are congruent to $b (\mmod p)$. Then we need to sum up this over all $b$ modulo $p$ such that $(a,b)\in M(p)$, then over all $q$ which are $a$ modulo $p$, then over all $a (\mmod p)$ such that there exist 
$b$ with $(a,b)\in M(p)$ and finally over all $p$. 
Since \eqref{eq:prime} applies, the first step gives
$$
 \frac{\pi(x/pq)}{\varphi(p)} \left(1+O\left(\frac{1}{\log_2 x}\right)\right)-\frac{\pi(x/pq(\log x))}{\varphi(p)}\left(1+O\left(\frac{1}{\log_2 x}\right)\right),
$$
which equals
\begin{equation}
\label{eq:logx/pq}
\frac{x}{pq \varphi(p)\log(x/pq)} \left(1+O\left(\frac{1}{\log_2 x}\right)\right).
\end{equation}
Note that 
$$
\log(x/pq)=\log x+O(\log y)=(\log x)\left(1+O\left(\frac{1}{\log_2 x}\right)\right),
$$
so because of the presence of the error term we can replace the factor $\log(x/pq)$ in the denominator in \eqref{eq:logx/pq} by $\log x$. Thus, the count so far is
$$
\frac{x}{\varphi(p) pq \log x} \left(1+O\left(\frac{1}{\log_2 x}\right)\right).
$$
Now we sum up over all $q\in [p^{\log_2 x},y]$ which are $q\equiv b (\mmod p)$. By the Abel summation formula,
we infer that
\begin{eqnarray*}
\sum_{\substack{p^{\log_2 x}\le q\le y\\ q\equiv b (\mmod p)}} \frac{1}{q}& = &\left(\frac{\pi(t;p,b)}{t} \Big|_{t=p^{\log_2 x}}^{t=y}\right)+\int_{p^{\log_2 x}}^{y} \frac{\pi(t;p,b)}{t^2} dt\\
& = &\frac{1+O(1/\log_2 x)}{\varphi(p)}\int_{p^{\log_2 x}}^{y} \frac{\pi(t)}{t^2} dt\\
&  = & \frac{1+O(1/\log_2 x)}{\varphi(p)}\int_{p^{\log_2 x}}^{y} \frac{1}{t\log t} \left(1+O\left(\frac{1}{\log t}\right)\right) dt\\
& = & \frac{1+O(1/\log_2 x)}{\varphi(p)}\left(\log_2 y-\log_2 (p^{\log_2 x})+O(1)\right).
\end{eqnarray*}
Note that 
$$
\log_2 y-\log_2 (p^{\log_2 x})=\log_2 x-\log_2 p+O(\log_3 x).
$$ 
Since $p\le z_1$, it follows that
$$
\log_2 x-\log_2 p\ge (\log_3 x)^2.
$$
Thus,
$$
\log_2 y-\log_2 (p^{\log_2 x})=(\log_2 x -\log_2 p)\left(1+O\left(\frac{1}{\log_3 x}\right)\right).
$$
Thus, we get 
$$
\sum_{\substack{p^{\log_2 x}\le q\le y\\ q\equiv 
b (\mmod p)}}\frac{1}{q}=\frac{1}{\varphi(p)} (\log_2 x-\log_2 p)\left(1+O\left(\frac{1}{\log_3 x}\right)\right).
$$
Hence, we get that for fixed $p$, $a$ and $b$, the number of such $n$ is 
$$
\frac{x}{p\varphi(p)^2(\log x)}(\log_2 x-\log_2 p) \left(1+O\left(\frac{1}{\log_3 x}\right)\right).
$$
Now we sum up over all $n(a)$ which, by definition, is the number of $b\in \{1,\ldots,p-1\}$ such that $(a,b)\in M(p)$, then over all the $a$ such that $n(a)>0$. 
Keeping in mind that 
$$
\sum_{1\le a\le p-1} n(a)=|M(p)|=
\alpha p^2+O(p),
$$
we obtain a contribution of 
$$
\frac{\alpha x}{p\log x} (\log_2 x-\log_2 p)\left(1+O\left(\frac{1}{\log_3 x}\right)\right)
\left(1+O\left(\frac{1}{p}\right)\right).
$$
Now we sum the 
latter expression up over all $p\in [z_1,y_1]$ and 
on using that $$1+O\left(\frac{1}{p}\right)=
1+O\left(\frac{1}{\log_3 x}\right)$$ in
that range and the fact that
$$
\sum_{p\le t} \frac{\log_2 p}{p}=\frac{1}{2} (\log_2 t)^2\left(1+O\left(\frac{1}{\log_2 t}\right)\right),
$$
we
get that the number of $r$ we are after is
\begin{equation}
\label{eq:final}
\alpha\left(\frac{x\log_2 x}{\log x}\sum_{z_1\le p\le y_1} \frac{1}{p}-\frac{x}{\log x}\sum_{z_1\le p\le y_1}\frac{\log_2 p}{p}\right)\left(1+O\left(\frac{1}{\log_3 x}\right)\right). 
\end{equation}
The first sum in \eqref{eq:final} above asymptotically equals 
$$
\log_2 y_1-\log_2 z_1+o(1)=\log_2 x\left(1+O\left(\frac{1}{\log_3 x}\right)\right).
$$
The second sum in \eqref{eq:final} is
$$
\frac{1}{2} \big((\log_2 y_1)^2-(\log_2 z_1)^2+O(\log_2 x)\big)  = \frac{(\log_2 x)^2}{2} \left(1+O\left(\frac{1}{(\log_3 x)^2}\right)\right).
$$
On putting everything together, the result is proved. 
\end{proof}
\section{Applications} 
\label{apps}
\subsection{Cyclotomic polynomials}

We define the {\it height} of a polynomial $f$ in $\mathbb Z[x]$, $h(f)$,  to be the
maximum of absolute value of the coefficients of $f$.
A polynomial of height one is said to be {\it flat}.\\
\indent The $n^{\text{th}}$ cyclotomic polynomial $\Phi_n$ is defined by
$$\Phi_n(x)=\prod_{1\le j\le n\atop (j,n)=1}(x-\zeta_n^j)=
\sum_{k=0}^{\varphi(n)}a_n(k)x^k,$$
 where $\varphi$ is Euler's totient function and $\zeta_n$ a primitive $n^{\text{th}}$ root of unity. 
 For a very readable introduction to the properties of coefficients of cyclotomic polynomials, the reader is referred to Thangadurai \cite{Thanga}.\\
\indent The coefficients $a_n(k)$ are integers that tend to be small. For example, for $n\le 104$ we have
$|a_n(k)|\le 1$, but $a_{105}(7)=-2$. Note that 105 is the smallest ternary integer.
It can be shown that if $|a_n(k)|>1$, then $n$ must have at least three distinct odd
prime factors (see \cite{Migotti}). 
For a more recent reproof see, e.g., 
Lam and Leung \cite{LL}.\\
\indent Gallot and Moree \cite{GM2} showed that the set
$\{a_n(k):0\le k\le \varphi(n)\}$ consists of a string of
consecutive integers in case $n$ is ternary.
Different proofs of this 
fact were given by Bachmann \cite{B4} and 
Bzd{\c e}ga \cite{BZ}. 
In all three 
papers \cite{B4,BZ,GM2} this was achieved by establishing that, in
case $n$ is ternary, $|a_n(k)-a_n(k-1)|\le 1$. Thus, neighboring coefficients differ by at most one. In 2014 Bzd{\c e}ga \cite{BZ2} went beyond
this and characterized all $k$ such that $|a_{pqr}(k)-a_{pqr}(k-1)|=1$
and determined the number of $k$'s for which this equality holds.
There are various papers
devoted to ternary cyclotomic polynomials, e.g. \cite{B1, B2, B3, Beiter, GM, GMW, Kaplan, ZZ1}.\\
\indent For a long time the main conjecture on ternary cyclotomic polynomials
was one made by Sister Marion Beiter 
in 1968.
\begin{conj}[Sister Beiter conjecture \cite{Beiter0}]
Let $p<q<r$ be primes.
The cyclotomic coefficient $a_{pqr}(k)$ 
satisfies
$|a_{pqr}(k)|\le (p+1)/2$. 
\end{conj} 
\noindent Sister Beiter herself established her
conjecture for $p=3$ and $p=5$ \cite{Beiter}. 
Zhao and Zhang \cite{ZZ1} proved it
for $p=7$. However, for every $p\ge 11$ the 
conjecture is false as was shown by
Gallot and Moree \cite{GM}. They  
put forward the following conjecture.
\begin{conj}[Corrected Sister Beiter conjecture, Gallot and Moree \cite{GM}]
\label{corrected}
Let $p<q<r$ be primes.
The cyclotomic coefficient $a_{pqr}(k)$ 
satisfies
$|a_{pqr}(k)|\le 2p/3$. 
\end{conj}
\noindent This conjecture is sharp as it becomes false
if the
ratio  $2/3$ is replaced by
any smaller number \cite{GM}. 
It has been shown to hold if the ratio $2/3$ is 
replaced by $3/4$ \cite{B1}.
\subsection{Flat cyclotomic polynomials}
Cyclotomic polynomials $\Phi_n$ are called flat if
$h(\Phi_n)=1$. The main challenge here is to find all
$n$ such that $\Phi_n$ is flat. For contributions,
see \cite{B3,Elder,Kaplan,Kaplan2}.
In particular, Broadhurst made a far reaching conjecture
here, cf. \cite{Kaplan2}.
Kaplan \cite{Kaplan} found the following family of
cyclotomic polynomials.
\begin{Thm}[Kaplan \cite{Kaplan}]
\label{thm:Kaplan}
If $p<q$ are primes and $r\equiv \pm 1(\mmod{pq})$, 
then $\Phi_{pqr}$ is flat.
\end{Thm}
\noindent Elder \cite{Elder} conjectured that if $n$ has
five or more odd prime factors, then
$\Phi_n$ is not flat. It thus seems that flat polynomials are
quite sparse.
\subsection{Inverse cyclotomic polynomials}
We define $\Psi_n(x)=(x^n-1)/\Phi_n(x)$ to be the $n^{\text{th}}$ {\it inverse} cyclotomic polynomial.
Since $x^n-1=\prod_{d|n}\Phi_d(x)$, we find that $\Psi_n=\prod_{d|n,d<n}\Phi_d$.
Thus, $\Psi_n$ is of degree $n - \phi(n)$ and has integer coefficients $c_{n}(k)$ which, like those of
the cyclotomic polynomials, tend to be small. For example
$\Psi_n$
has coefficients that are $\le 1$ in absolute value 
for $n\le 560$. Moreover, $c_{p}(k)\in \{-1,1\}$ and $c_{pq}(k) \in \{-1, 0,1\}$ 
(compare \cite[Lemma 5]{Mor}). 

We now recall two results on heights of cyclotomic and inverse cyclotomic polynomials 
due to Sister Beiter \cite{Beiter} and Moree \cite{Mor}. 
By the following result and the Prime Number Theorem for Arithmetic Progressions (a weaker
form of Theorem \ref{SW}), one infers
that the analogues of both the original 
and the corrected Sister Beiter conjecture for the ternary (inverse) cyclotomic polynomials are
true for $p=3$ and false for every $p\ge 5$.
\begin{Thm}[Moree  \cite{Mor}]
\label{main2}
Let $p<q<r$ be odd primes. Then
$h(\Psi_n)=p-1$ if and only if 
\begin{equation}
\label{bloebie}
q\equiv r\equiv \pm 1({\rm mod~}p)\text{~and~}r<{(p-1)\over (p-2)}(q-1).
\end{equation}
In the remaining cases, $h(\Psi_n)<p-1$.
\end{Thm} 
We say that a ternary
cyclotomic polynomial $\Psi_n$ is {\it coefficient optimal} if $h(\Psi_n)=P(n)-1$,
where $P(n)$ denote the smallest prime factor of $n$. Thus, a ternary integer $n=pqr$ is 
coefficient optimal if and only if $q$ and $r$ satisfy (\ref{bloebie}).
\subsection{Analytic results}
\label{sec:analytic}
\subsubsection{An analytic result related to
ternary inverse cyclotomic coefficients}
\label{subsection7.4.1}
On combining Theorem 
\ref{main2} with Theorem \ref{thm:assT}, the following 
result is obtained.
\begin{Thm} 
\label{co}
The number $N_{CO}(x)$ of ternary $n=pqr\le x$ such that $\Psi_n$ is
coefficient optimal satisfies
$$
N_{CO}(x)=C_1\frac{x}{(\log x)^2} +O\left(\frac{x\log\log x}{(\log x)^3}\right),
$$
with $C_1$ as in \eqref{logsum}.
\end{Thm}
\begin{cor} 
We have
$${N_{CO}(x)\over N_T(x)}\sim {2C_1\over (\log x)(\log \log x)^2}.$$
In particular, $\Psi_n$ is not coefficient optimal for almost all ternary $n$.
\end{cor} 

\begin{proof}
Combine Corollary \ref{cor1} and Theorem \ref{co}. 
\end{proof}
\subsubsection{Flatness}
\label{subsection7.4.2}
On combining Theorem 
\ref{thm::mod_pq} and Theorem \ref{thm:Kaplan}, the following result is obtained.
\begin{Thm}
\label{thm:flatness}
Let $F(x)$ denote the number of ternary $n\le x$ such
that $\Phi_n$ is flat. Then
$$
F(x)\ge (2C_2+o(1))\frac{x}{\log x},
$$
with $C_2$ as in \eqref{ceetwee}.
\end{Thm}
\subsubsection{The corrected Sister Beiter conjecture}
\label{subsection7.4.3}
The next result provides some evidence towards
the corrected Sister Beiter conjecture.
\begin{Thm}
\label{cycchal}
The number $N_{CB}(x)$ of ternary $n\le x$ such that $h(\Phi_n)\le 2P(n)/3$ satisfies
$$N_{CB}(x)\ge \Big({25\over 27}+o(1)\Big){x(\log \log x)^2\over 2\log x}.$$
\end{Thm}
\begin{cor}
\label{cor2}
The relative density of ternary integers for which
the correct Sister Beiter conjecture holds true is
at least $0.925$.
\end{cor}
The proof of Theorem \ref{cycchal} makes use of the following estimate due to Bzd\c{e}ga \cite{BZ}. For
completeness, we also consider what would happen if one would use an older estimate (2003)
due to Bachman \cite{B1}. In that case we obtain 
Theorem \ref{cycchal} and Corollary \ref{cor2}  with $25/27$ replaced by $8/9$ and
$0.925$ by $0.888$, respectively.
\begin{Thm} 
\label{GBmain}
Let $3\le p<q<r$ be primes. Let $q^*$ and $r^*$ be inverses
of $q$ and $r$ modulo $p$, respectively that satisfy $1\le q^*,r^*\le p-1$. Set
$a=\min(q^*,r^*,p-q^*,p-r^*)$
and let $1\le d\le p-1$ be defined by the relation
$adqr\equiv 1(\mmod p)$. Then we have (G. Bachman)
$$-\min\Big({p-1\over 2}+a,d\Big)\le a_{pqr}(k)\le \min\Big({p-1\over 2}+a,p-d\Big),$$
and (B. Bzd\c{e}ga)
$$-\min\Big(p+2a-d,d\Big)\le a_{pqr}(k)\le \min\Big(2a+d,p-d\Big).$$
\end{Thm}
It is not difficult to show that
$$d=\min(\max(q^*,r^*),\max(p-q^*,p-r^*)).$$
\begin{cor}
Put $d_1=\min(d,p-d)$.
We have (G. Bachman)
$$|a_{pqr}(k)|\le \min\Big({p-1\over 2}+a,p-d_1\Big),$$
and (B. Bzd\c{e}ga)
$$|a_{pqr}(k)|\le \min\Big(2a+d_1,p-d_1\Big).$$
\end{cor}
Let $1\le j,k\le p-1$ be integers. Put 
$$\alpha=\min(j,k,p-j,p-k),~~\delta=\min(\max(j,k),\max(p-j,p-k)),$$ and $\delta_1=\min(\delta,p-\delta)$.
Put $$GB(j,k)=\min\Big({p-1\over 2}+\alpha,p-\delta_1\Big){\rm ~and~}BB(j,k)=\min(2\alpha+\delta_1,p-\delta_1).$$ 
We can reformulate
the latter corollary in the following way.
\begin{cor}
\label{qstar}
If $q^*\equiv j(\mmod p)$ and $r^*\equiv k(\mmod p)$, then
$|a_{pqr}(k)|\le GB(j,k)$ and $|a_{pqr}(k)|\le BB(j,k)$.
\end{cor}
\begin{Def}
Put
$$GB(p)=\{(j,k):1\le j,k\le p-1,~GB(j,k)\le 2p/3\},$$ and
$$BB(p)=\{(j,k):1\le j,k\le p-1,~BB(j,k)\le 2p/3\}.$$
The
cardinality of $ GB(p)$ and $BB(p)$ we denote by $N_{GB}(p)$, respectively
$N_{BB}(p)$. 
\end{Def}
It is an elementary, but
quite tedious, exercise to evaluate
these quantities.
\begin{pro}
\label{BBpcard}
Let $p>3$ be a prime. Then
\begin{equation*}
N_{GB}(p)=
\begin{cases}
{8\over 9}p^2-{16\over 9}p+{8\over 9} & \text{if $p\equiv 
1(\mmod 3)$};\\
{8\over 9}p^2-{8\over 9}p-{16\over 9} & \text{if 
$p\equiv 2(\mmod 3)$}.
\end{cases}
\end{equation*}
and
\begin{equation*}
N_{BB}(p)=
\begin{cases}
{25\over 27}p^2-({8\over 27}({p\over 3})+2)p+{73\over 27} & \text{if $p\equiv \pm 2(\mmod 9)$};\\
{25\over 27}p^2-({8\over 27}({p\over 3})+2)p+{37\over 27} & \text{otherwise}.
\end{cases}
\end{equation*}
\end{pro}
\begin{proof}
We give a sketch. Note that if $(j,k)\in BB(p)$, then also $(k,j)\in  BB(p)$. 
It thus follows that 
$$N_{BB}(p)=2\sum_{\bfrac{1\le j<k\le p-1}{(j,k)\in BB(p)}}1
+\sum_{\bfrac{1\le j\le p-1}{(j,j)\in  BB(p)}}1.$$
Let us concentrate on the first sum as it is more complicated to evaluate.
We divide up the $(j,k)$ region $1\le j<k\le p-1$ into pieces on which
$BB(j,k)$ takes on a value not involving a minimum or maximum anymore and compare
this value with $2p/3$. Each of these contributions turns out to be a polynomial
in $p$ that is at most quadratic and has coefficients that depend at most on the
residue of $p$ modulo 9. Working out each of these contributions and summing gives
the required result. Alternatively, after one has established that the final answer
is a quadratic polynomial depending at most on the residue of $p$ modulo 9, one
finds the formula for $N_{BB}(p)$ be evaluting it for various values of $p$ and inferring
the coefficients of the polynomial from this.\\
\indent For $N_{GB}(p)$ we find similarly that the result should be a quadratic polynomial
depending at most on the residue of $p$ modulo 3.
\end{proof}
\begin{proof}[Proof of Theorem \ref{cycchal}]
Given an integer $a$ coprime to $p$, we write
$a^*$ for the inverse of $a$ modulo $p$ satisfying
$1\le a^*\le p-1$.
If $n=pqr$ satisfies $3\le p<q<r$ with
$q\equiv j^{*}(\mmod{p})$ and $r\equiv k^{*}(\mmod{p})$ 
and $(j,k)\in BB(p)$,
then $n$ satisfies the corrected Sister Beiter conjecture
by Corollary \ref{qstar}. By Proposition
\ref{BBpcard}, we have $N_{BB}(p)=25p^2/27+O(p)$.
Now apply
Theorem \ref{main} with $\alpha=25/27$ 
and $M(p)=\{(j^{*},k^{*}):(j,k)\in BB(p)\}$.
\end{proof}
\subsection{Applications in cryptography}
In ~\cite{gaps} by Camburu et al., there is a ternary counting
problem that is related to attempts of
Hong et al.~\cite{ate} to provide a simple and exact
formula for the minimum Miller loop length in the
${\rm Ate}_i$ pairing arising in elliptic curve cryptography.
The problem there is to estimate
$$\{pqr\le x:p<q<r,~4(p-1)>q,~p^2>r\}.$$
Also various other ternary counting problems are
considered in Camburu et al.~\cite{gaps}.

\section*{Acknowledgments}
The first  author was supported in part
 by NRF (South Africa) Grants CPRR160325161141 and an A-rated researcher award and by CGA (Czech Republic) Grant 17-02804S. Part of this work was done when this author visited 
 the Max Planck Institute for Mathematics in Bonn from January to July, 2017. 
\par The fourth author is supported by the Austrian Science Fund (FWF) : Project F5507-N26, which is part of the special Research Program ``Quasi Monte Carlo Methods : Theory and Application''. Part of this work was done while the author was supported by the Japan
Society for the Promotion of Science (JSPS) ``Overseas researcher under Postdoctoral Fellowship of JSPS''.

\medskip\noindent {\footnotesize School of Mathematics, University of the Witwatersrand, Private Bag X3, Wits 2050, South
Africa;\\
Department of Mathematics, Faculty of Sciences, University of Ostrava, 30. dubna 22, 701 03
Ostrava 1, Czech Republic;\\
Max-Planck-Institut f\"ur Mathematik,
Vivatsgasse 7, D-53111 Bonn, Germany.\\
e-mail: {\tt florian.luca@wits.ac.za}}

\medskip\noindent {\footnotesize Max-Planck-Institut f\"ur Mathematik,
Vivatsgasse 7, D-53111 Bonn, Germany.\\
e-mail: {\tt moree@mpim-bonn.mpg.de}}

\medskip\noindent {\footnotesize School of Mathematics and Statistics,
University College Dublin, Belfield, Dublin 4, Ireland. \\
e-mail: {\tt robert.osburn@ucd.ie}}

\medskip\noindent {\footnotesize Institute of Financial Mathematics and Applied Number Theory,
JKU Linz, Altenbergerstraße 69, 4040 Linz, Austria.\\
e-mail: {\tt sumaia.saad\_eddin@jku.at}}

\medskip\noindent {\footnotesize Max-Planck-Institut f\"ur Mathematik,
Vivatsgasse 7, D-53111 Bonn, Germany.\\
e-mail:  {\tt alisa.sedunova@phystech.edu} }


\begin{thebibliography}{00}

\bibitem{B1} G. Bachman, On the coefficients of ternary cyclotomic 
polynomials, {\it J. Number Theory} {\bf 100} (2003), 104--116.
\bibitem{B2} G. Bachman, Ternary cyclotomic polynomials with an optimally large set of coefficients, 
{\it Proc. Amer. Math. Soc.}  {\bf 132}  (2004),  1943--1950 (electronic).
\bibitem{B3} G. Bachman, Flat cyclotomic polynomials of order three,  
{\it Bull. London Math. Soc. } {\bf 38}  (2006),  53--60.
\bibitem{B4} G. Bachman, On ternary inclusion-exclusion 
polynomials, {\it Integers} {\bf 10} (2010), 623--638.
\bibitem{Beiter0} Sister M. Beiter, Magnitude of the coefficients of the cyclotomic polynomial 
$F_{pqr}(x)$, {\it Amer. Math. Monthly}  {\bf 75} (1968), 370--372.
\bibitem{Beiter} Sister M. Beiter, Coefficients of the cyclotomic polynomial 
$F_{3qr}(x)$, {\it Fibonacci Quart.}  {\bf 16}  (1978), 302--306.
\bibitem{Bom} E. Bombieri, Le grand crible dans la th\'eorie analytique des nombres, {\it Ast\'erisque} {\bf 18} (1987), 103 pp.
\bibitem{BZ} B. Bzd{\c e}ga, Bounds on ternary cyclotomic coefficients, 
{\it Acta Arith.} {\bf 144} (2010), 5--16.
\bibitem{BZ2} B. Bzd{\c e}ga, Jumps of ternary cyclotomic coefficients,
{\it Acta Arith.} {\bf 163} (2014), 203--213. 
 \bibitem{gaps} O.-M. Camburu, E.-A. Ciolan, F. Luca, P. Moree 
 and I.E. Shparlinski, Cyclotomic coefficients: gaps and jumps, 
{\it J. Number Theory} {\bf 163} (2016), 211--237. 
\bibitem{Elder} S. Elder, Flat cyclotomic polynomials: a new approach, arXiv:1207.5811.
\bibitem{ford} K. Ford, Vinogradov's integral and bounds for the Riemann zeta function, 
{\it Proc. London Math. Soc.} (3) {\bf 85} No 3
(2002), 565--633.

\bibitem{GM2} Y. Gallot and P. Moree, Neighboring ternary cyclotomic coefficients differ by at most one, 
{\it J. Ramanujan Math. Soc.} {\bf 24} 
(2009), 235--248.
\bibitem{GM} Y. Gallot and P. Moree, Ternary cyclotomic polynomials having a large
coefficient, {\it J. Reine Angew. Math.} {\bf 632} (2009), 105--125.
\bibitem{GMW} Y. Gallot and P. Moree and 
R. Wilms, The family of ternary cyclotomic polynomials 
with one free prime, 
{\it Involve} {\bf 4} (2011), 317--341. 
\bibitem{HW} G.H. Hardy 
and E.M. Wright, {\it An introduction to the theory of numbers}, Sixth edition, Revised by D. R. Heath-Brown and J. H. Silverman, With a foreword by Andrew Wiles, Oxford University Press, Oxford, 2008. 
\bibitem{H} A. Hildebrand, On the number of prime factors of an
integer, {\it Ramanujan revisited} (Urbana-Champaign, Ill., 1987),
167--185, Academic Press, Boston, MA, 1988.
\bibitem{ate} H. Hong, E. Lee, H.-S. Lee and C.-M. Park, 
Maximum gap in (inverse) cyclotomic polynomial, 
{\it J. Number Theory} {\bf 132} (2012), 2297--2315.
\bibitem{Kaplan} N. Kaplan, Flat cyclotomic polynomials of order three, {\it J. Number Theory} 
{\bf 127} (2007), 118--126.
\bibitem{Kaplan2} N. Kaplan, Flat cyclotomic polynomials of order four and higher, {\it Integers} {\bf 10} (2010), A30, 357--363. 
\bibitem{LL} T.Y. Lam and K.H. Leung, On the cyclotomic polynomial 
 $\Phi_{pq}(X)$, 
 {\it Amer. Math. Monthly} {\bf 103} (1996), 562--564.
 
\bibitem{L} E. Landau, {\it Handbuch der Lehre von der Verteilung der 
Primzahlen}, Chelsea Publishing Co., New York, 1953.

\bibitem{Migotti} A. Migotti, Zur Theorie der Kreistheilungsgleichung, 
Z. B. der Math.-Naturwiss, Class der Kaiserlichen Akademie der Wissenschaften, Wien {\bf 87} (1883), 7--14.
\bibitem{Mor} P. Moree, Inverse cyclotomic polynomials, {\it J. Number Theory} {\bf 129} (2009), 667--680.
\bibitem{MS} P. Moree and S. Saad Eddin, Products of two proportional 
primes,  {\it Int. J. Number Theory} {\bf 13} (2017), 2583--2596.
\bibitem{Tenen} G. Tenenbaum, {\it Introduction to analytic and probabilistic number 
theory}, Cambridge Studies in Advanced Mathematics {\bf 46}. Cambridge University Press, Cambridge, 1995.
\bibitem{Thanga} R. Thangadurai, On the coefficients of 
cyclotomic polynomials, {\it Cyclotomic fields and related topics} 
(Pune, 1999), 311--322, Bhaskaracharya
  Pratishthana, Pune, 2000. 
\bibitem{trudgian} T. Trudgian, Updating the error term in the prime number theorem, {\it Ramanujan J.} {\bf 39} (2016), 225--234.
\bibitem{ZZ1} J. Zhao and X. Zhang, Coefficients of ternary cyclotomic polynomials, {\it J. Number Theory} {\bf 130} (2010), 2223--2237.

\end{thebibliography}
\end{document}